\documentclass[10pt, twoside]{amsart}
\advance\oddsidemargin by -1.0cm
\advance\evensidemargin by -1.0cm
\advance\topmargin by -1.0cm 
\usepackage{amssymb}
\usepackage{amstext}
\usepackage{amscd}   
\theoremstyle{plain}
\newtheorem{cor}{Corollary}[section]
\newtheorem{lem}{Lemma}[section]
\newtheorem{thm}{Theorem}[section]
\newtheorem{prop}{Proposition}[section]
\theoremstyle{definition}

\newcommand{\C}{\ensuremath{\mathbb{C}}}
\newcommand{\R}{\ensuremath{\mathbb{R}}}

\newcommand{\N}{\ensuremath{\mathbb{N}}}

\newcommand{\Ric}{\ensuremath{\mathrm{Ric}}}

\newtheorem{proposition}{Proposition}[section]

\newtheorem{remark}[proposition]{Remark}

\newcommand{\be}{\begin{equation}}
\newcommand{\ee}{\end{equation}}
\newcommand{\bea}{\begin{eqnarray}}
\newcommand{\eea}{\end{eqnarray}}
\newcommand{\bean}{\begin{eqnarray*}}
\newcommand{\eean}{\end{eqnarray*}}

\begin{document}

\title[Eigenvalue estimates for the Dirac operator on 
K\"ahler-Einstein manifolds]{Eigenvalue estimates for the Dirac operator
on K\"ahler-Einstein manifolds of even complex dimension}

%-------------------------------------------
\address{%\hspace{-5mm} 
{\normalfont\ttfamily 
Institute of Mathematics \newline
Humboldt-Universit\"at zu Berlin\newline
Office: Rudower Chaussee 25\newline
D-10099 Berlin, Germany\newline
kirchber@mathematik.hu-berlin.de}}
\author[K.-D. Kirchberg]{K.-D. Kirchberg}
\maketitle 

\begin{center}
(\today)
\end{center}

\begin{abstract}
\noindent In K\"ahler-Einstein case of positive scalar curvature and even complex dimension, an
improved lower bound for the first eigenvalue of the Dirac operator is given. It is shown by a general
construction that there are manifolds for which this new lower bound itself is the first eigenvalue.\\

\thispagestyle{empty}



\end{abstract}
\bigskip

\noindent\section*{\bf Introduction}

For every eigenvalue $\lambda$ of the Dirac operator $D$ on a compact spin K\"ahler manifold
$M$ of positive scalar curvature $S$ and even (odd) complex dimension $m$, one knows the 
estimate
\begin{equation}\label{gl-01}
\lambda^2 \ge \frac{m}{4(m-1)} S_0 \quad \left(\lambda^2 \ge \frac{m+1}{4m} S_0 \right) \, , 
\end{equation}

where $S_0$ is the minimum of $S$ on $M$ \cite{3,4,5}. This estimate is sharps in the sense that
 there are manifolds for which (\ref{gl-01}) is an equality for the first eigenvalue of $D$. These
so called limiting manifolds were geometrically described by A.~Moroianu \cite{9,10}. For odd complex 
dimension $m$, the corresponding limiting manifolds are Einstein. However, for even $m \ge 4$, the
 limiting manifolds are not Einstein. This leads to the conjecture that the estimate (\ref{gl-01}) 
can be improved in K\"ahler-Einstein case of even complex dimension. In this paper we show that the 
conjecture is true. We obtain the estimate
\begin{equation}\label{gl-02}
 \lambda^2 \ge \frac{m+2}{4m} S
\end{equation}
which is better than (\ref{gl-01}) for even $m \ge 4$. We prove that this estimate is also sharp 
in the sense above. It is also known that the estimate (\ref{gl-01}) can be deduced from a more general
result which gives a lower bound of $\lambda^2$ if the eigenvalue $\lambda$ of $D$ is of determined type.
In the case considered here, we prove an analogous type depending estimate from which the estimate
 (\ref{gl-02}) follows immediately.
Our paper is organized as follows. In Section 1 we collect some basic notions and facts of
K\"ahlerian spin geometry. In particular, it is shown there that, for every eigenvalue $\lambda 
\not= 0$ of the Dirac operator $D$ on a K\"ahler manifold of complex dimension $m$, the corresponding
 eigenspace splits into $m$ subspaces. Using this fact the type of an eigenvalue $\lambda \not= 0$ 
is defined. 
In Section 1 the reader finds also some basic formulas which are used in the proofs of the estimates.
For the readers convenience, the up to now main results concerning the type depending lower bound for the
eigenvalues of the Dirac operator on compact K\"ahler manifolds of positive scalar curvature are shortly
written down in Section 2. In this section is also given a short proof of the basic type depending
estimate. So the reader can compare directly the known result and the corresponding method of the proof
with that in Section 3 which contains the main results of our paper. A special part of Section 3 
deals with the construction of limiting manifolds for the estimate (\ref{gl-02}). 

\section{Some basic notions and facts}

Let $M$ be a spin K\"ahler manifold of real dimension $n=2m$ with metric $g$, complex structure  $J$ 
and spinor bundle $\Sigma$. Then the corresponding K\"ahler form $\Omega$ defined by $\Omega (X,Y)
:= g(JX,Y)$ can be considered as an endomorphism of $\Sigma$ via Clifford multiplication. In this sense,
$\Omega$ acts on a spinor $\psi$ locally by
\begin{equation}\label{03}
\Omega \psi =\frac{1}{2} X_a \cdot J(X^a ) \cdot \psi \, ,
\end{equation}
where $(X_1, \ldots , X_n)$ is any local frame of vector fields and $(X^1, \ldots, X^n)$ the 
associated coframe  defined by $X^a := g^{ab} X_b , (g^{ab}):= ( g_{ab})^{-1}$ and
$g_{ab} := g(X_a , X_b)$. It follows that $\Omega$ then is antiselfadjoint with respect to the 
Hermitian scalar product $\langle \cdot , \cdot \rangle$ on $\Sigma$ and we have the well-known
 orthogonal splitting
\begin{equation}\label{gl-04}
\Sigma = \Sigma_0 \oplus \Sigma_1 \oplus \cdots \oplus \Sigma_m \, , 
\end{equation}
where $\Sigma_k$ denotes the eigensubbundle corresponding  to the eigenvalue \, $i (2k-m)$ of $\Omega$
with $\mbox{rank}_{\C} (\Sigma_k)= \left( \begin{array}{c}
 m \\ k \end{array} \right)$. 
For any real vector or vector field $X$ on $M$, we use the notations $p(X) := \frac{1}{2} (X- i
JX), \bar{p} (X):= \frac{1}{2} (X + iJX)$. Then we have the complex Clifford relations
\begin{equation} \label{gl-05}
\begin{array}{c}
p (X) \cdot \bar{p} (Y) + \bar{p} (Y) \cdot p (X) = - g (X,Y) + i \Omega (X,Y) \, , \\
\bar{p} (X) \cdot {p} (Y) + {p} (Y) \cdot \bar{p} (X) = - g (X,Y) - i \Omega (X,Y) \, , 
\end{array}
\end{equation}
\begin{equation} \label{gl-06}
\begin{array}{c}
p(X) \cdot p(Y) + p(Y) \cdot p(X) =0 \, , \\
\bar{p} (X) \cdot \bar{p} (Y) + \bar{p} (Y) \cdot \bar{p} (X)=0 \ .
\end{array}
\end{equation}
The proof of the following lemma is a simple calculation if one uses a local eigenframe of the 
endomorphism.

\begin{lem} \label{lem-1-1}
Suppose that $\alpha : TM \rightarrow TM$ is a selfadjoint endomorphism which commutes with the 
complex structure $J$. Then, for any local frame of vector fields $(X_1, \ldots , X_n)$,
the equations
\begin{equation} \label{gl-07}
\begin{array}{l}
p (\alpha (X_a)) \cdot \bar{p} (X^a)= p(X_a) \cdot \bar{p} (\alpha (X^a)) = - \frac{1}{2} \mathrm{tr}
(\alpha ) + \frac{i}{2} \alpha (X_a) \cdot J(X^a) \, , \\[0.5em]
\bar{p} (\alpha (X_a)) \cdot {p} (X^a)= \bar{p}(X_a) \cdot {p} (\alpha (X^a)) = - \frac{1}{2} \mathrm{tr}
(\alpha ) - \frac{i}{2} \alpha (X_a) \cdot J(X^a) 
\end{array}
\end{equation}
are valid.
\end{lem}

The Ricci form $\rho$ is defined by $\rho (X,Y) := g(J(X) , \Ric (Y))$.
Considered as an endomorphism of $\Sigma$ the action of $\rho$ on spinors is locally given by
\begin{equation}\label{gl-08}
\rho \psi = \frac{1}{2} \Ric (X_a) \cdot J(X^a) \cdot \psi \, .
\end{equation}

The application of Lemma \ref{lem-1-1} to the special cases $\alpha = \mathrm{id}$ and $\alpha =
\Ric$ yields the identities
\begin{equation} \label{gl-09}
\begin{array}{l}
p(X_a) \cdot \bar{p} (X^a) = - m + i \Omega \, , \\
\bar{p} (X_a) \cdot p(X^a)= - m - i \Omega \, , 
\end{array}
\end{equation}
\begin{equation} \label{gl-10}
\begin{array}{l}
p (\Ric (X_a)) \cdot \bar{p} (X^a)=p (X_a) \cdot \bar{p} (\Ric (X^a))= - \frac{S}{2} + i \rho \, , \\[0.5em]
\bar{p} (\Ric (X_a)) \cdot p(X^a) = \bar{p} (X_a) \cdot p (\Ric (X^a)) = - \frac{S}{2} - i \rho \, .
\end{array}
\end{equation}
We consider the operators $D, D_+, D_- : \Gamma (\Sigma) \to \Gamma (\Sigma)$ locally defined by 
$D \psi := X^a \cdot  \nabla_{X_a} \psi, D_+ \psi =p (X^a) \cdot \nabla_{X_a} \psi, D_-
\psi := \bar{p} (X^a) \cdot \nabla_{X_a} \psi$. Then $D$ is the Dirac operator and we have the 
well-known operator identities
\begin{equation}\label{gl-11}
D=D_+ + D_- \, , 
\end{equation}
\begin{equation} \label{gl-12}
D_+^2 =0 \quad , \quad D_-^2 =0 \, , 
\end{equation}
\begin{equation}\label{gl-13}
D^2 =D_+ \circ D_- + D_- \circ D_+ \, . 
\end{equation}
Moreover, for all $k \in \{ 0,1, \ldots , m\}$, it holds that 
\begin{equation}\label{gl-14}
D_+ (\Gamma (\Sigma_k)) \subseteq \Gamma (\Sigma_{k+1}) \quad , \quad D_- (\Gamma (\Sigma_k)) \subseteq
\Gamma (\Sigma_{k-1}) \, , 
\end{equation}
where we here and in the following use the convention that $\Sigma_k \subset \Sigma$ is the zero 
subbundle if $k \not\in \{0,1, \ldots , m \}$.
We remember that $\Sigma$ is furnished with a canonical antilinear structure $j : \Sigma \to \Sigma$.
$j$ is parallel, commutes with the Clifford multiplication by real vectors and has further the 
properties
\begin{equation}\label{gl-15}
\langle j \varphi , j \psi \rangle = \langle \psi, \varphi \rangle \, , 
\end{equation}
\begin{equation}\label{gl-16}
j(\Sigma_k)= \Sigma_{m-k} \quad (k= 0,1, \ldots , m) \, , 
\end{equation}
\begin{equation}\label{gl-17}
j^2 =(-1)^{\frac{m(m+1)}{2}} \, . 
\end{equation}
Moreover, the relations
\begin{equation}\label{gl-18}
D \circ j = j \circ D \quad , \quad D_{\pm} \circ j = j \circ D_{\mp}
\end{equation}
are valid.
Let $\lambda$ be an eigenvalue of $D$ and $E^{\lambda} (D)$ the corresponding eigenspace. According to 
(\ref{gl-04}) every eigenspinor $\psi \in E^{\lambda} (D)$ decomposes in the form
$\psi = \psi_0 + \psi_1 + \cdots + \psi_m$ with $\psi_k \in \Gamma (\Sigma_k)$ for all 
$k\in \{ 0,1, \ldots , m \}$. By (\ref{gl-11}) and (\ref{gl-14}), then the eigenvalue equation $D{\psi} = \lambda
\psi$ is equivalent to the system of equations
\begin{equation}\label{gl-19}
D_+ \psi_{k-1} + D_- \psi_{k+1} \ = \ \lambda \psi_k \quad (k= 0,1, \ldots , m) \, . 
\end{equation}
For $k\in \{ 1,2, \ldots , m \}$ and $\psi \in E^{\lambda} (D)$, we consider the spinor
\begin{displaymath}
e^{\lambda}_k \psi := \frac{1}{\lambda} (D_- \psi_k + D_+ \psi_{k-1}) \ . 
\end{displaymath}
\begin{lem}\label{lem-1-2}
It holds that
\begin{equation}\label{gl-20}
D e^{\lambda}_k \psi \ = \ \lambda e^{\lambda}_k \psi \quad (k= 1,2, \ldots , m) \ . 
\end{equation}
\end{lem}
\begin{proof}
If we apply $D=D_+ + D_-$ to the equations (\ref{gl-19}), then we obtain the system
\begin{displaymath}
D_- D_+ \psi_{k-1} + D_+ D_- \psi_{k+1} = \lambda \cdot D_- \psi_k + \lambda \cdot D_+ \psi_k
\quad (k=0,1, \ldots, m) \ ,
\end{displaymath}
which is equivalent to
\begin{equation} \label{gl-21}
\left.
\begin{array}{l}
D_- D_+ \psi_{k-1} = \lambda D_- \psi_k  \\[0.5em]
D_+ D_- \psi_{k+1} = \lambda D_+ \psi_k 
\end{array} \right\} \quad (k=0,1, \ldots m ) \ . 
\end{equation}
It follows
\begin{eqnarray*}
D e^{\lambda}_k \psi & \stackrel{(\ref{gl-11}), (\ref{gl-12})}{=} & \frac{1}{\lambda} (D_+ D_- \psi_k + D_- D_+
\psi_{k-1} ) \stackrel{(\ref{gl-21})}{=} \frac{1}{\lambda} ( \lambda \cdot D_+ \psi_{k-1} +
\lambda D_- \psi_k) \\[0.5em]
&=& \lambda e^{\lambda}_k \psi  \ . 
\end{eqnarray*}
\end{proof}
By Lemma \ref{lem-1-2}, we obtain endomorphisms
\begin{displaymath}
e^{\lambda}_k : E^{\lambda} (D)  \to E^{\lambda} (D) \quad (k=1,2, \ldots , m)
\end{displaymath}
for every eigenvalue $\lambda \not= 0$. The next lemma one proves by similar calculations.
\begin{lem}\label{lem-1-3}
For every eigenvalue $\lambda \not= 0$ of $D$, the endomorphisms $e^{\lambda}_k \ 
(k=1,2, \ldots , m)$ have the following properties:\\
\begin{itemize}
\item[(i)] \quad $(e^{\lambda}_k )^2 = e^{\lambda}_k \quad (k=1,2, \ldots , m) \ , $\\
\item[(ii)] \quad $e^{\lambda}_k \circ e^{\lambda}_l =0 \quad (k \not= l)$ \ , \\
\item[(iii)] \quad $e^{\lambda}_1 + e^{\lambda}_2 + \cdots + e^{\lambda}_m = \mathrm{id} \ . $
\end{itemize}
\end{lem}
This lemma implies the following proposition immediately.
\begin{prop}\label{prop-1-4}
Let $\lambda \not= 0$ be any eigenvalue of the Dirac operator $D$ on a K\"ahler manifold $M$ of complex
dimension $m$. Then the corresponding eigenspace $E^{\lambda} (D)$ splits into $m$ subspaces
\begin{equation}\label{gl-22}
E^{\lambda} (D) = E^{\lambda}_1 (D) \oplus E^{\lambda}_2 (D) \oplus \cdots \oplus
E^{\lambda}_m (D) \ , 
\end{equation}
where the subspace $E^{\lambda}_k (D):= e^{\lambda}_k (E^{\lambda} (D))$ is characterized as follows:
Every eigenspinor $\psi \in E^{\lambda}_k (D)$ is of the form $\psi = \psi_{k-1} + \psi_k$ such 
that the components $\psi_{k-1} \in \Gamma (\Sigma_{k-1})$ and $\psi_k \in \Gamma (\Sigma_k)$ 
satisfy the equations
\begin{equation}\label{gl-23}
D_+ \psi_{k-1} = \lambda \psi_k \quad , \quad D_- \psi_k = \lambda \psi_{k-1} \ , 
\end{equation}
\begin{equation}\label{gl-24}
D_- \psi_{k-1} = 0 \quad , \quad D_+ \psi_k = 0 \ , 
\end{equation}
\begin{equation}\label{gl-25}
D^2 \psi_{k-1} = \lambda^2 \psi_{k-1} \quad , \quad D^2 \psi_k = \lambda^2 \psi_k
\  .
\end{equation}
Moreover, the subspaces $E^{\lambda}_k (D)$ are related by
\begin{equation}\label{gl-26}
j(E^{\lambda}_k (D)) = E^{\lambda}_{m-k+1} (D) \quad (k=1,2, \ldots , m) \ . 
\end{equation}
\end{prop}
We remark that the relation (\ref{gl-26}) follows from (\ref{gl-16}), (\ref{gl-18})
and (\ref{gl-23}).\\
Now, let $M$ be compact. Then, for any $\varphi, \psi \in \Gamma (\Sigma)$,
we use the notations
\begin{displaymath}
(\varphi, \psi) := \int\limits_M \langle \varphi , \psi \rangle \omega \quad
, \quad \| \psi \| := \sqrt{(\psi , \psi )} \ , 
\end{displaymath}
where $\omega := \frac{1}{m!} \Omega^m$ is the volume form. It is well-known 
that then $D_+$ and $D_-$ are adjoint to each other with respect to this Hermitian
$L^2$-scalar-product, i.e., it holds that
\begin{equation}\label{gl-27}
(D_{\pm} \varphi, \psi )=( \varphi, D_{\mp} \psi ) \ . 
\end{equation}
\begin{prop}\label{prop-1-5}
Let $M$ be a compact spin K\"ahler manifold of complex dimension $m$ and let
$\lambda \not= 0$ be any eigenvalue of the Dirac operator $D$. Then we have 
the following:
\begin{itemize}
\item[(i)] The corresponding decomposition (\ref{gl-22}) is orthogonal with 
respect to the $L^2$-scalar-product $( \cdot , \cdot)$.\\
\item[(ii)] For every $k \in \{ 1,2, \ldots , m \}$ and every eigenspinor
$\psi = \psi_{k-1} + \psi_k \in E^{\lambda}_k (D)$, the components
$\psi_{k-1} \in \Gamma (\Sigma_{k-1})$ and $\psi_k \in \Gamma (\Sigma_k)$ have
the same length
\begin{equation}\label{gl-28}
\| \psi_{k-1} \| = \| \psi_k \| \ . 
\end{equation}
\end{itemize}
\end{prop}
\begin{proof}
For every $k \in \{ 1,2, \ldots , m\}$ and any $\varphi, \psi \in \Gamma (\Sigma)$, it holds
that 
\begin{eqnarray*}
(e^{\lambda}_k \varphi , \psi ) = \frac{1}{\lambda} (D_- \varphi_k + D_+ \varphi_{k-1} , \psi )
&=& \frac{1}{\lambda} \Big(( D_- \varphi_k, \psi_{k-1} )+ (D_+ \varphi_{k-1} , \psi_k)\Big) 
\stackrel{(\ref{gl-27})}{=}\\[0.5em]
\frac{1}{\lambda} \Big(( \varphi_k , D_+ \psi_{k-1})+ (\varphi_{k-1} , D_- \psi_k )\Big) &=&
\frac{1}{k} \Big(( \varphi , D_+ \psi_{k-1} ) + \varphi , D_- \psi_k )\Big)= \\[0.5em]
\Big( \varphi , \frac{1}{\lambda} (D_+ \psi_{k-1} + D_- \psi_k ) \Big) &=& (\varphi , e^{\lambda}_k 
\psi ) \ . 
\end{eqnarray*}
Thus, $e^{\lambda}_k$ is selfadjoint. This proves the assertion (i).\\
Finally, we have
\begin{displaymath}
\lambda \| \psi_{k-1} \|^2 = (\lambda \psi_{k-1} , \psi_{k-1} ) \stackrel{(\ref{gl-23})}{=} 
(D_- \psi_k , \psi_{k-1} ) \stackrel{(\ref{gl-27})}{=} (\psi_k , D_+ \psi_{k-1})
\stackrel{(\ref{gl-23})}{=}(\psi_k , \lambda \psi_k )= \lambda \| \psi_k \|^2 \ . 
\end{displaymath}
This implies (\ref{gl-28}) since $\lambda \not= 0$.
\end{proof}
It is not excluded that in the splitting (\ref{gl-22}) some of the subspaces $E^{\lambda}_k (D)$ are
trivial. For any eigenvalue $\lambda \in \mathrm{Spec} (D) - \{ 0 \}$, we define the type of
$\lambda$ by 
\begin{displaymath}
\mathrm{typ} (\lambda) := \min  \{ k \in \{ 1,2, \ldots , m  \} \ | \ E^{\lambda}_k (D) \not= 0 
\} \ . 
\end{displaymath}
Then (\ref{gl-26}) implies
\begin{equation}\label{gl-29}
1 \le \mathrm{typ} (\lambda) \le \left[ \frac{m+1}{2} \right] \ , 
\end{equation}
where $[ \cdot ]$ denotes the integer part. Thus, we have a map
\begin{displaymath}
\mathrm{typ} : \mathrm{Spec} (D) - \{ 0 \} \to \left\{ 1,2, \ldots , \left[ \frac{m+1}{2} \right] \right\} \ . 
\end{displaymath}
For $k \in \left\{ 1,2, \ldots , \left[ \frac{m+1}{2} \right] \right\}$, we define $\mathrm{Spec_k}
(D) := \mathrm{typ}^{-1} (\{ k \} )$.

\section{The known results}

In this section we shortly describe the known main results concerning lower estimates for the 
first eigenvalues of all types of the Dirac operator on compact K\"ahler manifolds with positive scalar
curvature. Let $M$ be a spin K\"ahler manifold of dimension $n = 2m$. For $k \in \{ 0,1, \ldots , m \}$, 
we consider the K\"ahlerian twistor operator of degree $k$ \cite{11}
\begin{displaymath}
\mathcal{D}^{(k)} : \Gamma (\Sigma_k ) \to \Gamma (TM \otimes \Sigma_k )
\end{displaymath}
locally defined by $\mathcal{D}^{(k)} \psi := X^a \otimes \mathcal{D}^{(k)}_{X_a} \psi$ with
\begin{displaymath}
\mathcal{D}^{(k)}_X \psi := \nabla_X \psi + \frac{1}{2(k+1)} \bar{p} (X) \cdot D_+ \psi + 
\frac{1}{2(m-k+1)}p (X) \cdot D_- \psi \ . 
\end{displaymath}
We see that $\psi \in \Gamma (\Sigma_k)$ is in the kernel of $\mathcal{D}^{(k)}$ if it satisfies the 
equation
\begin{equation}\label{gl-30}
\nabla_X \psi + \frac{1}{2(k+1)} \bar{p} (X) \cdot D_+ \psi + \frac{1}{2(m-k+1)} p (X) 
\cdot D_- \psi =0
\end{equation}
for every real vector field $X$. The elements of $\mathrm{ker} (\mathcal{D}^{(k)})$ are
called K\"ahlerian twistor spinors of degree $k$. We remark that (\ref{gl-30}) is equivalent
to the two equations
\begin{equation}\label{gl-31}
\nabla_{\bar{p} (X)} \psi + \frac{1}{2(k+1)} \bar{p} (X) \cdot D_+ \psi =0 \ , 
\end{equation}
\begin{equation}\label{gl-32}
\nabla_{p(X)} \psi + \frac{1}{2(m-k+1)} p (X) \cdot D_- \psi =0 \ . 
\end{equation}
The K\"ahlerian twistor operator
\begin{displaymath}
\mathcal{D} : \Gamma ( \Sigma) \to \Gamma (TM \otimes \Sigma )
\end{displaymath}
is then defined by $\mathcal{D} := \mathcal{D}^{(0)} \oplus \mathcal{D}^{(1)} \oplus
\ldots \oplus \mathcal{D}^{(m)}$.\\
Thus, if $\psi = \psi_0 + \psi_1 + \cdots + \psi_m$ is the decomposition of $\psi \in \Gamma
(\Sigma)$ according to (\ref{gl-04}), then we have by definition
\begin{equation} \label{gl-33}
\mathcal{D} \psi = \mathcal{D}^{(0)} \psi_0 + \mathcal{D}^{(1)} \psi_1 + \cdots + 
\mathcal{D}^{(m)} \psi_m \ . 
\end{equation}
It is easy to see that, for all $\psi \in \Gamma (\Sigma)$ and every local frame of vector fields
$(X_1 , \ldots , X_m)$, the equations
\begin{equation} \label{gl-34}
p (X^a)  \cdot \mathcal{D}_{X_a} \psi =0 \quad , \quad \bar{p} (X^a)\cdot \mathcal{D}_{X_a} 
\psi =0
\end{equation}
are satisfied. This implies
\begin{equation} \label{gl-35}
X^a \cdot \mathcal{D}_{X_a} \psi =0 \ , 
\end{equation}
i.e., the image of $\mathcal{D}$ is contained in the kernel of the Clifford multiplication.\\
Using the orthogonal decomposition (\ref{gl-33}) one proves the following proposition by a 
straightforward calculation \cite{11}.
\begin{prop} \label{prop-2-1}
For any $\psi \in  \Gamma (\Sigma)$, we have the equation
\begin{equation}\label{gl-36}
|  \mathcal{D} \psi |^2 = \sum\limits^m_{k=0} \left( | \nabla \psi_k |^2 - \frac{1}{2(k+1)} | D_+ \psi_k |^2 - \frac{1}{2(m-k+1)} | D_- \psi_k |^2 \right) \ , 
\end{equation}
where $\psi = \psi_0 + \psi_1 + \cdots + \psi_m$ is the decomposition according to (\ref{gl-04}).
\end{prop}
\begin{thm} \label{thm-2-2}
Let $M$ be a compact spin K\"ahler manifold of dimension $n=2m$ and let $S_0 > 0$ be the minimum
of the scalar curvature $S$ on $M$. Then, for all $k\in \{ 1,2,\ldots , \left[ \frac{m+1}{2} \right] \}$ 
and every $\lambda \in \mathrm{Spec}_k (D)$, we have the estimate
\begin{equation}\label{gl-37}
\lambda^2 \ge \frac{k}{4k-2} \ S_0 \ . 
\end{equation}
\end{thm}
\begin{proof}
Let $\lambda \in \mathrm{Spec}_k (D)$ and let $\psi = \psi_{k-1} + \psi_k \in 
E^{\lambda}_k (D)$ be a corresponding eigenspinor. Inserting the component $\psi_{k-1}$ into
(\ref{gl-36}) and using (\ref{gl-23}), (\ref{gl-24}) we obtain  the equation\\[0.5em]
(*) \hfill $\displaystyle  | \mathcal {D} \psi_{k-1} |^2 = | \nabla \psi_{k-1} |^2 - \frac{\lambda^2}{2k}
| \psi_k |^2 \ . $ \hfill \mbox{}\\[0.5em]
On the other hand, using the Schr\"odinger-Lichnerowicz formula
\begin{equation}\label{gl-38}
\nabla^* \nabla = D^2 - \frac{S}{4}
\end{equation}
and (\ref{gl-25}) we find the equation\\[0.5em]
(2*) \hfill $\displaystyle \nabla^* \nabla \psi_{k-1} = \left( \lambda^2 - \frac{S}{4} \right) 
\psi_{k-1} \ . $ \hfill \mbox{}\\[0.5em]
Now, integrating the equation (*) we have
\begin{eqnarray*}
&& \| \mathcal{D} \psi_{k-1} \|^2 = \| \nabla \psi_{k-1} \|^2 - \frac{\lambda^2}{2k} \| \psi_k\| =
(\nabla^* \nabla \psi_{k-1} ) - \frac{\lambda^2}{2k} \| \psi_k \|^2
\stackrel{(2*)}{=} \\[0.5em]
&& \left( \Big( \lambda^2 - \frac{S}{4} \Big) \  \psi_{k-1} , \psi_{k-1} \right) - \frac{\lambda^2}{2k} 
\| \psi_k \|^2 \le \\[0.5em]
&& \Big( \lambda^2 - \frac{S_0}{4} \Big) \| \psi_{k-1} \|^2 - \frac{\lambda^2}{2k} \| \psi_k \|^2 
\stackrel{(\ref{gl-28})}{=} \left( \frac{2k-1}{2k} \lambda^2 - \frac{S_0}{4} \right) \| \psi_{k-1} 
\|^2 \ . 
\end{eqnarray*}
Thus, we obtain the inequality
\begin{equation}\label{gl-39}
\| \mathcal{D} \psi_{k-1} \|^2 \le \left( \frac{2k-1}{2k} \, \lambda^2 - \frac{S_0}{4} \right)
\| \psi_{k-1} \|^2
\end{equation}
which immediately implies (\ref{gl-37}).
\end{proof}
\begin{remark}\label{rem-2-3}
An analogous calculation as in the proof of Theorem \ref{thm-2-2} shows that, for the 
component $\psi_k$ of $\psi = \psi_{k-1} + \psi_k \in E^{\lambda}_k (D)$, the inequality
\begin{equation} \label{gl-40}
\| \mathcal{D} \psi_k \|^2 \left( \frac{2(m-k+1)-1}{2(m-k+1)} \ \lambda^2 - \frac{S_0}{4} \right)
\, \| \psi_k \|^2
\end{equation}
is valid.
\end{remark}
\begin{remark} \label{rem-2-4}
We remember that we have the estimate
\begin{equation} \label{gl-41}
\lambda^2 \ge \frac{n}{4(n-1)} \, S_0
\end{equation}
for every eigenvalue $\lambda$ of the Dirac operator $D$ on a compact Riemannian spin manifold of
dimension $n$ with positive scalar curvature $S$ \cite{1}. Thus, the suppositions of Theorem
\ref{thm-2-2} imply that there are no harmonic spinors $(0 \not\in \mathrm{Spec} (D))$. 
Hence, since the lower bound in (\ref{gl-37}) decreases if the type $k$ of the eigenvalue $\lambda$
increases, Theorem \ref{thm-2-2} immediately implies the estimate (\ref{gl-01}).
\end{remark}
\begin{remark} \label{rem-2-5}
Let $\lambda^{(k)}_1 \in \mathrm{Spec}_k (D)$ denote the first eigenvalue of type $k$.
Then in the limiting case of (\ref{gl-37}) we have
\begin{equation}\label{gl-42}
\lambda^{(k)}_1 = \sqrt{\frac{k}{4k-2} \, S}
\end{equation}
since the scalar curvature has to be constant in the limiting case. Furthermore, (\ref{gl-39}) implies
the equation $\mathcal{D} \psi_{k-1}=0$ which is equivalent to the equation
\begin{equation}\label{gl-43}
\nabla_X \psi_{k-1} + \frac{\lambda^{(k)}_1}{2k} \, \bar{p} (X) \cdot \psi_k = 0
\end{equation}
for every real vector field $X$. In particular, it follows immediately that $\psi_{k-1}$ is antiholomorphic
\begin{equation}\label{gl-44}
\nabla_{p(X)} \, \psi_{k-1} =0 \ . 
\end{equation}
\end{remark}
Thus, the limiting case of (\ref{gl-37}) is characterized in general by the existence of a special
antiholomorphic K\"ahlerian twistor spinor of degree $k-1$.
\begin{remark}\label{rem-2-6}
In the special limiting case with $k= \frac{m+1}{2}$ (limiting case of (\ref{gl-01}) with odd $m$), the
first eigenvalue $\lambda_1$ of $D$ is given by
\begin{equation}\label{gl-45}
\lambda_1 = \sqrt{\frac{m+1}{4m} \, S}
\end{equation}
and (\ref{gl-43}) takes the special form
\begin{equation}\label{gl-46}
\nabla_X \psi_{k-1} + \frac{\lambda_1}{m+1} \, \bar{p} (X) \cdot \psi_k =0 \ . 
\end{equation}
\end{remark}
Since $m-k+1=k$, (\ref{gl-40}) additionally implies the equation $\mathcal{D} \psi_k=0$ which
is equivalent to
\begin{equation}\label{gl-47}
\nabla_X \psi_k + \frac{\lambda_1}{m+1} \, p(X) \cdot \psi_{k-1} =0 \ . 
\end{equation}
In particular, $\psi_k$ is holomorphic
\begin{equation}\label{gl-48}
\nabla_{\bar{p} (X)} \, \psi_k =0 \ . 
\end{equation}
The equations (\ref{gl-46}), (\ref{gl-47}) show that, by definition, $\psi = \psi_{k-1} + \psi_k$ is
a K\"ahlerian Killing spinor. Conversely, the existence of a K\"ahlerian Killing spinor on a 
K\"ahler manifold $M$ with positive scalar curvature implies that $M$ is Einstein of odd complex dimension $m$ and has the limiting property
\cite{4,8}.\\

The first results concerning the classification of limiting manifolds for the estimate 
(\ref{gl-01}) have been proved in complex dimensions $m=2$ and $m=3$, where the only
limiting manifolds up to equivalence are $S^2 \times S^2, S^2 \times T^2$ \cite{2} and
$\C P^3, F(\C^3)$ \cite{6}, respectively ($T^2$ denotes the flat torus and $F(\C^3)$ the flag
manifold.). The general classification was given by A.~Moroianu in \cite{9} for $m$ odd and in
\cite{10} for $m$ even. We collect his results in the following theorem.
\begin{thm} \label{thm-2-7}
(i) In odd complex dimensions $m=4l+1$, the only limiting manifold for the estimate (\ref{gl-01})
is the complex projective space $\C P^m$. In odd complex dimensions $m=4l+3$, the limiting
manifolds of (\ref{gl-01}) are just the twistor spaces over quaternionic K\"ahler manifolds of 
positive scalar curvature.\\
(ii) A K\"ahler manifold $M$ of even complex dimension $m \ge 4$ is a limiting manifold of (\ref{gl-01})
if and only if its universal cover is isometric to a Riemannian product $N \times \R^2$, where 
$N$ is a limiting manifold of (\ref{gl-01}) for the odd complex dimension $m-1$ and $M$ is the 
suspension over a flat parallelogram of two commuting isometries of $N$ preserving  a K\"ahlerian
Killing spinor.
\end{thm}
Finally, let us consider the limiting case of the estimate (\ref{gl-37}) for the first eigenvalue 
of type $k$ of $D$ with $1 < k < \frac{m+1}{2}$. By a result of M.~Pilca (see the proof of
Theorem 5.15 in \cite{11}), the corresponding limiting manifolds can not be Einstein. This
situation leads to the question if the estimate (\ref{gl-37}) can be improved in K\"ahler-Einstein 
case in which the type $k$ of the eigenvalue satisfies the condition $1 < k < \frac{m+1}{2}$
which includes the case of even complex dimension $m \ge 4$. A positive answer to this question is
given in the next section.

\section{The K\"ahler-Einstein case}

Let $M$ be a spin K\"ahler manifold of dimension $n=2m$. The holomorphic (antiholomorphic)
part$\nabla^{1,0} \psi (\nabla^{0,1} \psi)$ of the covariant derivative $\nabla \psi$ of a 
spinor $\psi \in  \Gamma (\Sigma)$ is locally defined by
\begin{displaymath}
\nabla^{1,0} \psi := g(X^a) \otimes \nabla_{p(X_a)} \psi = g (\bar{p}(X^a)) \otimes 
\nabla_{X_a} \psi
\end{displaymath}
\begin{displaymath}
(\nabla^{0,1} \psi := g(X^a) \otimes \nabla_{\bar{p}(X_a)} \psi = g ({p}(X^a)) \otimes 
\nabla_{X_a} \psi ) \ , 
\end{displaymath}
where $(X_1 , \ldots , X_n)$ is a local frame of real vector fields and, for a complex vector
field $Z$, $g(Z)$ denotes the complex 1-form given by $(g(Z))(W):= g(Z,W)$. By definition, 
then we have $\nabla \psi = \nabla^{1,0} \psi + \nabla^{0,1} \psi$ with 
\begin{displaymath}
\nabla^{1,0} \psi \in \Gamma (\Lambda^{1,0} \otimes \Sigma) \ , \nabla^{0,1} \psi \in \Gamma
(\Lambda^{0,1} \otimes \Sigma )
\end{displaymath}
and, for any real vector field $X$, it holds that $\nabla^{1,0}_X = \nabla_{p(X)}$, 
$\nabla^{0,1}_X = \nabla_{\bar{p} (X)} \cdot \psi$ is said to be  holomorphic (antiholomorphic)
if $\nabla^{0,1} \psi =0 \ (\nabla^{1,0} \psi =0)$. For any complex vector field $Z,W$, we use the 
notation 
\begin{displaymath}
\nabla^2_{Z,W} := \nabla_Z \circ \nabla_W - \nabla_{\nabla_Z W}
\end{displaymath}
for the corresponding tensoriel second order covariant derivative. We consider the two
K\"ahler-Bochner-Laplacians
\begin{displaymath}
\nabla^{1,0*} \nabla^{1,0} \ , \ \nabla^{0,1*} \nabla^{0,1} \ : \Gamma (\Sigma) \rightarrow
\Gamma (\Sigma)
\end{displaymath}
locally defined by 
\begin{displaymath}
\nabla^{1,0*} \nabla^{1,0} := - \nabla^2_{\bar{p} (X_a), p(X^a)} \ , \ \nabla^{0,1*} \nabla^{0,1} :=
- \nabla^2_{p(X_a) , \bar{p} (X^a)} \ . 
\end{displaymath}
Obviously, these Laplacians and the Bochner Laplacian $\nabla^* \nabla := - \nabla^2_{X_a, X^a}$
of Riemannian spin geometry are related by
\begin{equation}\label{gl-49}
\nabla^{1,0*} \nabla^{1,0} + \nabla^{0,1*} \nabla^{0,1} = \nabla^* \nabla \ . 
\end{equation}
Moreover, we have the operator identities
\begin{equation}\label{gl-50}
2 \nabla^{1,0*} \nabla^{1,0} = D^2 - \frac{S}{4} - \frac{i}{2} \rho \ , 
\end{equation}
\begin{equation}\label{gl-51}
2 \nabla^{0,1*} \nabla^{0,1} = D^2 - \frac{S}{4} + \frac{i}{2} \rho \ , 
\end{equation}
where $D$ is the Dirac operator, $S$ the scalar curvature and $\rho$ the Ricci form. A proof of these
formulas one finds in \cite{7}, Section 4.
\begin{thm}\label{thm-3-1}
Let $M$ be a spin K\"ahler-Einstein manifold of dimension $n=2m$ with positive scalar curvature $S$. 
Then, for every $k \in \{ 1,2, \ldots ,  [ \frac{m+1}{2} ] \}$ and every 
eigenvalue $\lambda \in \mathrm{Spec}_k (D)$, the estimate
\begin{equation}\label{gl-52}
\lambda^2 \ge \frac{m-k+1}{2m} \, S 
\end{equation}
is valid.
\end{thm}
\begin{proof}
We remark firstly that our suppositions imply that $M$ is compact. 
Moreover, the Einstein condition $\Ric = \frac{S}{n} \mathrm{id}$ implies
that the Ricci form is given by\\[0.5em]
(*) \hfill $\displaystyle \rho = \frac{S}{n} \, \Omega \ . $ \hfill \mbox{}\\[0.5em]
Now, let $\lambda \in \mathrm{Spec}_k (D)$ and let $\psi = \psi_{k-1} + \psi_k \in E^{\lambda}_k (D)$ be a corresponding eigenspinor. Then, by integration of
the function $2 |  \nabla^{1,0} \psi_{k-1} |^2$, we obtain
\begin{displaymath}
0 \le 2 \| \nabla^{1,0} \psi_{k-1} \|^2 = (2 \nabla^{1,0*} \nabla^{1,0} 
\psi_{k-1} , \psi_{k-1}) \stackrel{(\ref{gl-50})}{=}
\end{displaymath}
\begin{displaymath}
\left( \Big(D^2 - \frac{S}{4} - \frac{i}{2} \rho \Big) \psi_{k-1} , \psi_{k-1} \right)
\stackrel{(\ref{gl-25}), (*)}{=} \left( \Big( \lambda^2 - \frac{S}{4} - 
\frac{i}{4m} S \Omega \Big) \psi_{k-1} , \psi_{k-1} \right) =
\end{displaymath}
\begin{displaymath}
\left( \Big(\lambda^2 - \frac{S}{4} + \frac{S}{4m} (2(k-1) - m)\Big) \psi_{k-1} , 
\psi_{k-1} \right) = \Big(\lambda^2 - \frac{m-k+1}{2m} \, S \Big) \| \psi_{k-1} \|^2
\end{displaymath}
and, hence, the estimate (\ref{gl-52}).
\end{proof}
Theorem \ref{thm-3-1} immediately implies the following corollary.
\begin{cor}\label{cor-3-2}
If $M$ is a compact spin K\"ahler-Einstein manifold of even complex dimension
$m$ with positive scalar curvature $S$, then we have the estimate (\ref{gl-02})
for every eigenvalue $\lambda$ of $D$.
\end{cor}
\begin{remark}\label{rem-3-3}
The estimate (\ref{gl-52}) improves (\ref{gl-37}) in K\"ahler-Einstein case if 
the type $k$ of the eigenvalue satisfies the condition $1 < k < \frac{m+1}{2}$, 
in particular, if $k= \frac{m}{2}$ with even $m \ge 4$.
\end{remark}
\begin{remark}\label{rem-3-4}
The proof of Theorem \ref{thm-3-1} shows that in the limiting case of (\ref{gl-52})
the component $\psi_{k-1}$ of the eigenspinor $\psi = \psi_{k-1} + \psi_k
\in E^{\lambda}_k (D)$ must be antiholomorphic $(\nabla^{1,0} \psi_{k-1} =0)$.
Conversely, it is known that in K\"ahler-Einstein case every antiholomorphic
spinor $\varphi \in \Gamma (\Sigma_{k-1})$ satisfies the equation
\begin{equation}\label{gl-53}
D^2 \varphi = \frac{m-k+1}{2m} \, S \varphi
\end{equation}
(see \cite{7}, Proposition 10). Hence, the limiting case of (\ref{gl-52}) is 
characterized by the existence of an antiholomorphic section in the subbundle
$\Sigma_{k-1} \subset \Sigma$ or, equivalently, by the existence of a 
holomorphic section in $\Sigma_{m-k+1}$ if we take into account (\ref{gl-26}).
\end{remark}
By Remark \ref{rem-3-4}, we immediately obtain our next theorem.
\begin{thm} \label{thm-3-5}
Let $M$ be a spin K\"ahler-Einstein manifold of positive scalar curvature $S$ and 
complex dimension $m$. Then the inequality (\ref{gl-52}) is an equality for the 
first eigenvalue of the type $k$ of the Dirac operator $D$ if and only if
the bundle $\Sigma_{k-1} (\Sigma_{m-k+1})$ admits an antiholomorphic
(holomorphic) section. In particular, for even $m$, (\ref{gl-02}) is an 
equality for the first eigenvalue of $D$ if and only if the bundle 
$\Sigma_{\frac{m-2}{2}} (\Sigma_{\frac{m+2}{2}})$ has an 
antiholomorphic (holomorphic) section.
\end{thm}
By a construction of limiting manifolds, we show now that (\ref{gl-02}) is 
also a sharp estimate. The following theorem gives a certain construction
principle.
\begin{thm} \label{thm-3-6}
If $M_1$ and $M_2$ are limiting manifolds for the estimate (\ref{gl-01}) of
odd complex dimensions $m_1$ and $m_2$, respectively, such that its Ricci
tensors have the same positive eigenvalue, then the product $M_1 \times 
M_2$ is a limiting manifold of the estimate (\ref{gl-02}).
\end{thm}
\begin{proof}
By supposition, $M_1$ and $M_2$ are compact K\"ahler-Einstein manifolds with 
positive scalar curvatures $S_1$ and $S_2$, respectively, such that 
$S_1 /2m_1 = S_2 / 2m_2$. Hence, the product $M: = M_1 \times M_2$ is also
a compact K\"ahler-Einstein of even complex dimension $m:= m_1 + m_2$ with
scalar curvature $S:= S_1 + S_2$. Now, we remember the following general
fact.If
\begin{displaymath}
\Sigma^1 = \bigoplus\limits^{m_1}_{k=0} \, \Sigma^1_k \quad , \quad 
\Sigma^2 = \bigoplus\limits^{m_2}_{l=0} \, \Sigma^2_l
\end{displaymath}
are the spinor bundles of $M_1$ and $M_2$, respectively, then the spinor bundle 
$\Sigma$ of $M$ is of the form
\begin{equation}\label{gl-54}
\Sigma = \Sigma^1 \otimes \Sigma^2 = \bigoplus\limits^m_{k=0} \Sigma_k \ , 
\end{equation}
where the subbundle $\Sigma_k$ is given by
\begin{equation}\label{gl-55}
\Sigma_k = \bigoplus\limits^k_{l=0} \left( \Sigma^1_l \otimes \Sigma^2_{k-l}
\right) \ . 
\end{equation}
Moreover, if $\nabla^1, \nabla^2$ and $\nabla$ denote the covariant derivatives 
on $\Sigma^1, \Sigma^2$ and $\Sigma$, respectively, then $\nabla_X$ acts on the 
tensor product $\varphi \otimes \psi \in \Gamma (\Sigma)$ of spinors
$\varphi \in \Gamma (\Sigma^1)$ and $\psi \in \Gamma (\Sigma^2)$ by
\begin{equation}\label{gl-56}
\nabla_X (\varphi \otimes \psi )=( \nabla^1_{X_1} \varphi ) \otimes \psi +
\varphi \otimes (\nabla^2_{X_2} \psi ) \ , 
\end{equation}
where $X=X_1 + X_2$ is the orthogonal decomposition of the vector field
 $X$ according to the canonical splitting of the tangent bundle
\begin{equation}\label{gl-57}
TM = TM_1 \oplus TM_2 \ . 
\end{equation}
Since $M_1$ and $M_2$ are limiting manifolds for the estimate (\ref{gl-01}), 
there exist antiholomorphic spinors $\psi_1 \in \Gamma (\Sigma^1_{\frac{m_1 -1}{2}})$
and $\psi_2 \in \Gamma (\Sigma^2_{\frac{m_2 -1}{2}})$ according to Theorem 
\ref{thm-3-5}. By (\ref{gl-56}), we see that the spinor $\psi_1 \otimes \psi_2
\in \Gamma (\Sigma^1_{\frac{m_1 -1}{2}}) \otimes \Sigma^2_{\frac{m_2 -1}{2}})
\subseteq \Gamma (\Sigma_{\frac{m-2}{2}})$ is antiholomorphic too. Hence, by Theorem
\ref{thm-3-5}, $M$ is a limiting manifold for the estimate (\ref{gl-02}).
\end{proof}
Using the results of Section 2 and Theorem \ref{thm-3-6} we obtain the 
following corollary which lists some special examples of limiting manifolds.
\begin{cor} \label{cor-3-7}
For all $k,l \in \N \cup \{ 0 \}$, the K\"ahler-Einstein manifolds
$\C P^{4k+1} \times \C P^{4l+1}, \C P^3 \times \C P^{4k+1}$ and 
$F(\C^3) \times \C P^{4k+1}$ are limiting  manifolds of the estimate
(\ref{gl-02}). Moreover, $\C P^3 \times F(\C^3)$ is a limiting manifold.
\end{cor}

\vspace{1cm}

%\refname{References}

\end{document}